\documentclass[12pt]{amsart}

\usepackage{amssymb}

\newtheorem{theorem}{Theorem}[section]
\newtheorem{lemma}[theorem]{Lemma}
\newtheorem{proposition}[theorem]{Proposition}
\newtheorem{corollary}[theorem]{Corollary}
\newtheorem{assumption}[theorem]{Assumption}

\numberwithin{equation}{section}

\begin{document}
\baselineskip=15pt

\title[On a question of Sean Keel]{On a question of Sean Keel}

\author[I. Biswas]{Indranil Biswas}

\address{School of Mathematics, Tata Institute of Fundamental
Research, Homi Bhabha Road, Bombay 400005, India}

\email{indranil@math.tifr.res.in}

\author[S. Subramanian]{S. Subramanian}

\address{School of Mathematics, Tata Institute of Fundamental
Research, Homi Bhabha Road, Bombay 400005, India}

\email{subramnn@math.tifr.res.in}

\subjclass[2000]{14F05, 14J26}

\keywords{Ample line bundle, finite field, Nakai-Moishezon criterion}

\date{}

\begin{abstract}
Let $X$ be a smooth projective surface defined over $\overline{{\mathbb 
F}_p}$, and let $L$ be a line bundle over $X$ such that $L\cdot Y\,
>\, 0$ for every complete curve $Y$ contained in $X$. A question of Keel 
asks 
whether $L$ is ample. If $X$ is a ${\mathbb P}^1$--bundle over a curve, 
we prove that this question has an affirmative answer.

\end{abstract}

\maketitle

\section{Introduction}\label{sec1}

The Nakai-Moishezon criterion says that a line bundle $L$ on a smooth
projective surface $X$ defined over an algebraically closed field is
ample if and only if $L\cdot Y\, >\, 0$ for every complete curve $Y$
contained in $X$ and $L\cdot L\, >\, 0$. It should be mentioned that
if $L\cdot Y\, >\, 0$ for every complete curve $Y\, \subset\,
X$, then $L\cdot L\, \geq\, 0$.
Mumford constructed a smooth complex projective surface 
$X$ and a line bundle $L\, \longrightarrow\, X$ such that
$L\cdot Y\, >\, 0$ for every complete curve $Y$ contained
in $X$ but $L$ is not
ample \cite[p. 56, Example 10.6]{Ha1}. The surface $X$ in
Mumford's example is a ${\mathbb P}^1_{\mathbb C}$--bundle over a
smooth projective curve.

Fix a prime $p$. Let $X$ be a smooth projective
surface defined over $\overline{{\mathbb F}_p}$, and let $L$
be a line bundle over $X$, such that $L.Y\, >\, 0$
for every complete curve $Y\, \subset\, X$. A question of Keel asks 
whether
$L$ is ample \cite[p. 3959, Question 0.9]{Ke}. See also \cite{To}
for related material.

Our aim here is to give an affirmative answer to the question of
Keel under the assumption that $X$ is a ${\mathbb P}^1$--bundle
over a curve.

We prove the following theorem (see Theorem \ref{thm1}):

\begin{theorem}\label{thm0}
Let $C$ be an irreducible smooth projective curve defined over
$\overline{{\mathbb F}_p}$, and let $E\,\longrightarrow\, C$
be a vector bundle of rank two. Let $L\,\longrightarrow\, {\mathbb 
P}(E)$ be a line bundle such that $L.Y\, >\, 0$
for every complete curve $Y$ contained in ${\mathbb P}(E)$. Then
$L$ is ample.
\end{theorem}

\section{Line bundles over a ruled surface}

Fix a prime $p$. Let $C$ be an irreducible smooth projective curve
defined over $\overline{{\mathbb F}_p}$. Let $E\,\longrightarrow\, C$ 
be a vector bundle of rank two. Let
$$
f_0 \, :\, {\mathbb P}(E)\, \longrightarrow\, C
$$
be the projective bundle for $E$. Let
$$
L\, \longrightarrow\, {\mathbb P}(E)
$$
be a line bundle.

There is a unique integer $n$ and a unique
line bundle $\xi_0\longrightarrow C$ such that
\begin{equation}\label{e1}
L\,=\, {\mathcal O}_{{\mathbb P}(E)}(n)\otimes f^*_0\xi_0\, .
\end{equation}

\begin{lemma}\label{lem1}
There is an irreducible smooth projective curve $M$ over
$\overline{{\mathbb F}_p}$ and a nonconstant morphism
$$
\varphi\, :\, M\, \longrightarrow\, C
$$
such that ${\rm degree}(\varphi^*\xi_0)$ is a multiple
of $n$, and ${\rm degree}(\varphi)$ is even.
\end{lemma}

\begin{proof}
This is a standard fact. The power of $p$ in the factorization
of $2n$ can be handled using an iteration of the
Frobenius morphism of $C$; separable morphisms are available for
other factors.
\end{proof}

Since $\text{degree}(\varphi^*E)\,=\, {\rm degree}(\varphi)\cdot
{\rm degree}(E)$ is even, and $\text{Pic}^0(M)$ is divisible,
there is a vector bundle of rank two
$$
V\, \longrightarrow\, M
$$
with $\bigwedge^2 V\,=\, {\mathcal O}_M$
such that $\varphi^*{\mathbb P}(E)\,=\,
{\mathbb P}(V)$. We fix such a vector bundle $V$. Let
\begin{equation}\label{e2}
\beta\, :\, {\mathbb P}(V)\,=\, \varphi^*{\mathbb P}(E)
\,\longrightarrow\, {\mathbb P}(E)
\end{equation}
be the natural morphism. Let
\begin{equation}\label{e3}
f\,:\, {\mathbb P}(V)\, \longrightarrow\, M
\end{equation}
be the natural projection.

\begin{assumption}\label{a1}
{\rm For every complete curve $Y\,\subset\,
{\mathbb P}(E)$, the inequality $L\cdot Y \, >\, 0$ holds.}
\end{assumption}

{}From Assumption \ref{a1} it follows
that the integer $n$ in \eqref{e1} is positive.

Since ${\rm degree}(\varphi^*\xi_0)$ is a multiple
of $n$, and ${\rm Pic}^0(M)$ is divisible
there is a line bundle $\xi$ on $M$ such that
\begin{equation}\label{e4}
\beta^*L \, =\, ({\mathcal O}_{{\mathbb P}(V)}(1)\otimes
f^*\xi)^{\otimes n}\, ,
\end{equation}
where $\beta$ and $f$ are constructed in \eqref{e2} and \eqref{e3}
respectively.
The morphism $\beta$ is finite. Therefore, from Assumption \ref{a1}
and \eqref{e4} it follows that or every complete curve $Y\,\subset\,
{\mathbb P}(V)$,
\begin{equation}\label{e5}
({\mathcal O}_{{\mathbb P}(V)}(1)\otimes f^*\xi)\cdot Y\, >\, 0\, .
\end{equation}

Assumption \ref{a1} implies that $L\cdot L\, \geq\,0$.
If $L\cdot L\, >\, 0$, then from a criterion of Nakai-Moishezon 
it follows that $L$ is ample \cite[p. 365, Theorem 1.10]{Ha2}.
We assume that
\begin{equation}\label{e6}
L\cdot L\,=\, 0\, .
\end{equation}
Our aim is to show that this assumption leads to a contradiction.

\begin{proposition}\label{prop1}
The degree of the line bundle $\xi\,\longrightarrow\, M$
in \eqref{e4} is zero.
\end{proposition}

\begin{proof}
{}From \eqref{e6} and \eqref{e4} it follows that 
$$
({\mathcal O}_{{\mathbb P}(V)}(1)\otimes
f^*\xi)\cdot ({\mathcal O}_{{\mathbb P}(V)}(1)\otimes
\xi)\,=\, 0\, .
$$
So
\begin{equation}\label{e7}
{\mathcal O}_{{\mathbb P}(V)}(1)\cdot 
{\mathcal O}_{{\mathbb P}(V)}(1) + 2
({\mathcal O}_{{\mathbb P}(V)}(1)\cdot (f^*\xi))\, =\, 0\, .
\end{equation}
But ${\mathcal O}_{{\mathbb P}(V)}(1)\cdot 
{\mathcal O}_{{\mathbb P}(V)}(1)\,=\,
\text{degree}(V)\,=\, 0$, and
$$
{\mathcal O}_{{\mathbb P}(V)}(1)\cdot f^*\xi\, =\,
\text{degree}(\xi)\, .
$$
Hence the proposition follows from \eqref{e7}.
\end{proof}

A line bundle on a projective variety is called
\textit{nef} if the degree of
its restriction to every complete curve is
nonnegative.

Proposition \ref{prop1} and \eqref{e5} together
give the following corollary:

\begin{corollary}\label{cor1}
The line
bundle ${\mathcal O}_{{\mathbb P}(V)}(1)$ is nef.
\end{corollary}

Let $F_M\, :\, M\, \longrightarrow\, M$ be the absolute
Frobenius morphism. A vector bundle $W$ over $M$ is
called \textit{strongly semistable} if $(F^i_M)^*W$ is
semistable for every $i\, \geq\, 1$.

\begin{proposition}\label{prop2}
The vector bundle $V$ over $M$ is strongly semistable.
\end{proposition}

\begin{proof}
Let $\widetilde{M}$ be an irreducible smooth projective
curve, and let $h\, :\, \widetilde{M}\,\longrightarrow\,
M$ be a morphism. Let
$$
h^*V\, \longrightarrow\, Q\, \longrightarrow\, 0
$$
be a quotient line bundle. Let
$$
\gamma\, :\, \widetilde{M}\,\longrightarrow\,
{\mathbb P}(V)
$$
be the morphism corresponding to $Q$. So
$Q\,=\, \gamma^*{\mathcal O}_{{\mathbb P}(V)}(1)$.
Now, from Corollary \ref{cor1},
$$
\text{degree}(Q)\,\geq \,0\,=\, \text{degree}(V)\, .
$$
Hence $V$ is strongly semistable.
\end{proof}

Fix a closed point $x_0\, \in\, M$. Let
$\varpi(M,x_0)$ be the fundamental group-scheme. We recall
that $\varpi(M,x_0)$ is constructed using the neutral
Tannakian category defined by the essentially finite vector
bundles on $M$ (see \cite{No} for essentially finite vector
bundles and fundamental group-scheme).
Using Proposition \ref{prop2} and \cite[p. 70, Theorem
3.2]{Su} we conclude that $V$ is given by a homomorphism
$$
\rho\, :\, \varpi(M,x_0)\, \longrightarrow\, \text{SL}(2,
\overline{{\mathbb F}_p})
$$
(recall that $\bigwedge^2 V\,=\, {\mathcal O}_M$).
In other words, the vector bundle $V$ is essentially finite.
This implies that there is a smooth projective curve
$\widetilde{M}$, and a morphism 
$$
h\, :\, \widetilde{M}\,\longrightarrow\, M\, ,
$$
such that the vector bundle $h^*V$ is trivial \cite[p. 557]{BH}.

Fix an isomorphism $\widetilde{M}\times {\mathbb 
P}^1_{\overline{{\mathbb F}_p}} \,
\stackrel{\sim}{\longrightarrow}\,
h^*{\mathbb P}(V)$. Fix a closed point
$c_0\, \in\, {\mathbb 
P}^1_{\overline{{\mathbb F}_p}}$. The image of the composition map
$$
\widetilde{M}\times \{c_0\} \,\longrightarrow\,
h^*{\mathbb P}(V)  \,\longrightarrow\, {\mathbb P}(V)
$$
contradicts the inequality in \eqref{e5}; recall that
$\text{degree}(\xi)\,=\, 0$ (see Proposition \ref{prop1}).

Therefore, we have proved the following theorem:

\begin{theorem}\label{thm1}
Under Assumption \ref{a1}, the line bundle
$L\, \longrightarrow\, {\mathbb P}(E)$ is ample.
\end{theorem}

%%%%%%%%%%%%%%%%%%%%%%%%%%%%%%%%%%%%%%%%%%%%%%%%%%%%%%%%%%%%%%%%%%%

\end{document}